\documentclass[11pt,a4paper,reqno]{amsart}
\usepackage{amsfonts}
\usepackage{amsmath,amssymb,amsthm,amsxtra}
\usepackage{mathrsfs}
\usepackage{float}
\usepackage[colorlinks, linkcolor=blue!72,anchorcolor=orange,
citecolor=red,urlcolor=Emerald, bookmarksopen,bookmarksdepth=2]{hyperref}
\usepackage[usenames,dvipsnames]{xcolor}
\usepackage{enumitem}
\usepackage{geometry,array}
\usepackage{graphicx}
\usepackage{subfigure}
\usepackage{bookmark}
\usepackage{tikz}
\usepackage{longtable}

\usepackage{bbm}

\usetikzlibrary{positioning, shapes.geometric}
\usepackage{hyperref}
\usepackage{cleveref}

\usepackage{url}

\usetikzlibrary{matrix,positioning,decorations.markings,arrows,decorations.pathmorphing, 
backgrounds,fit,positioning,shapes.symbols,chains,shadings,fadings,calc}
\tikzset{->-/.style={decoration={ markings, mark=at position #1 with
{\arrow{>}}},postaction={decorate}}}
\tikzset{-<-/.style={decoration={ markings, mark=at position #1 with
{\arrow{<}}},postaction={decorate}}}
\usepackage[all]{xy}

\usepackage{pdflscape}

\usepackage[markup=underlined]{changes}
\definechangesauthor[color=NavyBlue]{Zy}
\definechangesauthor[color=Emerald]{Qy}

\theoremstyle{plain}
\newtheorem{theorem}{Theorem}[section]
\newtheorem{lemma}[theorem]{Lemma}
\newtheorem{corollary}[theorem]{Corollary}
\newtheorem{proposition}[theorem]{Proposition}
\theoremstyle{definition}
\newtheorem{definition}[theorem]{Definition}

\newtheorem{remark}[theorem]{Remark}

\numberwithin{equation}{section}

\numberwithin{equation}{section}

\newtheorem{conjecture}[theorem]{Conjecture}

\usetikzlibrary{decorations.markings}
\tikzset{->-/.style={decoration={ markings, mark=at position #1 with
{\arrow{>}}},postaction={decorate}}}
\tikzset{-<-/.style={decoration={ markings, mark=at position #1 with
{\arrow{<}}},postaction={decorate}}}

\def\lp2{\mathcal O_{\mathbb P^2}(-3)}
\def\p2{\mathbb P^2}
\def\p3{\mathbb P^3}
\def\Stab{\mathrm{Stab}}

\def\num{\mathrm{num}}
\def\Ker{\mathrm{Ker}}

\def\Coh{\mathrm{Coh}}

\def\Geo{\mathrm{Geo}}

\def\p3{\mathbb P^3}
\def\gldim{\mathrm{gldim}}
\def\gl2{\widetilde{\mathrm{GL}}^+_{2}(\mathbb R)}
\def\stabg{\mathrm{Stab}^{\mathrm{Geo}}(\mathbb P^3)}
\def\staba{\mathrm{Stab}^{\mathrm{Alg}}(\mathbb P^3)}

\def\stabgc{\mathrm{Stab}^{\mathrm{Geo}}_{\triangle}(\mathbb P^3)}

\def\stabh{\mathrm{Stab}_H}
\def\stabgh{\mathrm{Stab}^{\mathrm{Geo}}_H}

\def\stabp{\mathrm{Stab}^{\dagger}(\mathbb P^3)}

\def\Cohb{\mathrm{Coh}^{\beta}}

\def\cA{\mathcal A}
\def\cB{\mathcal B}

\def\cD{\mathcal D}
\def\cE{\mathcal E}

\def\cO{\mathcal O}
\def\cP{\mathcal P}

\newcommand{\RR}{\mathbb{R}}
\newcommand{\PP}{\mathbb{P}}

\DeclareMathOperator{\ch}{ch}

\DeclareMathOperator{\GL}{GL}

\def\stabgp{\mathrm{Stab}^{\mathrm{Geo}\dagger}(\mathbb P^3)}

\usepackage{interval}
\usepackage{physics}

\begin{document}
\title{Stability Conditions on $\mathbb P^3$}

\author{Dongjian Wu}
\address{Department of Mathematical Sciences, Tsinghua University, 100084 Beijing, China}
\email{wdj20@mails.tsinghua.edu.cn}

\author{Nantao Zhang}
\address{Department of Mathematical Sciences, Tsinghua University, 100084 Beijing, China}
\email{znt21@mails.tsinghua.edu.cn}

\date{}

\keywords{Bridgeland stability conditions, Polarized threefolds, Bogomolov-Gieseker inequality, Global dimension function}


\maketitle


\begin{abstract}
We construct a subset of the space of stability conditions for any projective threefold with an ample polarization that satisfies a certain Bogomolov-Gieseker inequality to refine the result in \cite{BMS}. Then, we demonstrate that the global dimension, as defined in \cite{Qiu20,IQ1}, is 3 for any stability condition on $\p3$ constructed in \cite{BMS}. Finally, we formulate a conjecture concerning the contractibility of a principal connected component of $\mathrm{Stab}(\p3)$.

\end{abstract}

\maketitle
\tableofcontents

\setlength\parindent{0pt}
\setlength{\parskip}{5pt}

\section*{Introduction}
In this paper, we first construct an explicit subset of the space of stability conditions for any projective threefold with an ample polarization that fulfills a certain Bogomolov-Gieseker inequality, thereby refining the result in \cite{BMS}. This is achieved by employing the function in \Cref{Psi} motivated by the Le Potier function as discussed in \cite{FLZ,D22}. Then, we show that the global dimension of the subspace of stability conditions for $\p3$ constructed in \cite{BMS} is 3, following the approach in \cite{BMT,BMS,Moz22}. Finally, we propose a conjecture regarding the contractibility of a principal connected component of $\mathrm{Stab}(\p3)$, inspired by \cite{Li17}.
\subsection{Stability conditions on polarized threefolds}
The concept of stability conditions on triangulated categories was initially introduced by Bridgeland \cite{B07}, inspired by the fields of string theory and mirror symmetry. The existence of stability conditions on three-dimensional varieties is widely regarded as one of the most difficult problems in the theory of Bridgeland stability conditions. A strategy for constructing stability conditions on polarized threefolds was first introduced in \cite{BMT}. This approach involves the use of so-called tilt stability conditions to build geometric stability conditions, provided that a generalized Bogomolov-Gieseker inequality is satisfied. Since then, this existence problem for threefolds has been extensively studied, with a wealth of research including \cite{M14,BMS, BMSZ, Piy17, Kos18, Li19, Kos20, Kos22, Liu22}. The concept of algebraic stability conditions was introduced by \cite{M07b}, which has been employed to demonstrate the existence of geometric stability conditions on $\mathbb P^n$, as shown in \cite{Mu21}. 

Our first result is the construction of a subset within the the space of stability conditions for any projective threefold with an ample polarization that meets a certain Bogomolov-Gieseker inequality to refine the construction in \cite{BMS}. Recall that \cite{BMS} asserts the existence of a continuous open embedding:
\begin{equation*}
\begin{split}
\Sigma:\gl2\times\mathfrak{B}&\to\stabgh(X)\\
(g,(\alpha,\beta,a,b))&\mapsto\sigma_{\alpha,\beta}^{a,b}[g],
\end{split}
\end{equation*}
where $\mathfrak{B}:=\left\{(\alpha,\beta,a,b)\in\mathbb R^4\mid\alpha>0,a>\frac{\alpha^2}{6}+\frac{\alpha}{2}|b|\right\}$, and $\stabgh(X)$ represents the space of geometric stability conditions on $\cD^b(X)$ with respect to $(\Lambda_H,\lambda_H)$, as discussed in \Cref{sec:construct}. The pivotal approach to achieve the refinement is to introduce the following function motivated by the Le Potier function as discussed in \cite{FLZ,D22}:
\begin{definition}[\Cref{Psi}]
Let $X$ be a projective threefold with an ample polarization $H$. We define the function twisted by $\beta H$, $\Psi_{X,\nu}:\mathbb R_{>0}\times\mathbb R^2\to\mathbb R\cup\{-\infty\}$, with respect to a parameter $\nu\in\mathbb R$, as follows:
\[
\Psi_{X,\nu}(\alpha, \beta, b):=  \limsup_{\mu \to \nu} \left\{\frac{\ch_{3}^{\beta}(F) - bH \ch_{2}^{\beta}(F)}{H^{2} \ch_{1}^{\beta}(F)}:
{\begin{aligned}
&F\in\Cohb(X) \text{ is } \nu_{\alpha, \beta}\text{-semistable}\\ 
&\text{ with } \nu_{\alpha, \beta}(F) = \mu
\end{aligned}}
\right\},
\]
if the limit exists, and $\Psi_{X,\nu}(\alpha,\beta,b):=-\infty$ otherwise. In particular, when $\nu=0$, we denote by $\Psi_{X}:=\Psi_{X,0}$ for simplicity.
\end{definition}
Our first result is stated as follows:
\begin{theorem}[\Cref{main1}]
Let $X$ be a smooth projective threefold with an ample polarization $H$ satisfying the generalized Bogomolov-Gieseker inequality in \Cref{Bog}. Then, there is a continuous open embedding 
\begin{equation*}
\begin{split}
\Sigma_{\Psi}: \gl2\times\left\{(\alpha,\beta,a,b)\in\mathbb R^4\mid \alpha>0,a>\max\{\frac{\alpha^2}{6},\Psi_X(\alpha,\beta,b)\}\right\}&\to\stabgh(X)\\
(g,(\alpha,\beta,a,b))&\mapsto(Z_{\alpha, \beta}^{a, b}[g], \cA_{\alpha, \beta}[g]).
\end{split}
\end{equation*}
Furthermore, if $\Psi_{X}(\alpha,\beta,b)\ge\frac{\alpha^2}{6}$ for all $(\alpha,\beta,b)\in\mathbb R_{>0}\times\mathbb R^2$, we have 
\[
\Stab^{\mathrm{Geo}\dagger}_{H}(X)\cong\gl2 \times\left\{(\alpha,\beta,a,b)\in\mathbb R^4\mid \alpha>0,a>\Psi_X(\alpha,\beta,b)\right\},
\]
as $\Stab^{\Geo\dagger}_H(X)$ is specified in \Cref{stabghd}. 
\end{theorem}
From \Cref{bound}, we observe that $\Psi_X(\alpha,\beta,b)\le\frac{\alpha^2}{6}+\frac{\alpha}{2}|b|$. Hence, this result provides a refinement of the construction in \cite{BMS}. Furthermore, we study these stability condition spaces from a homological standpoint. The global dimension function on stability condition spaces, as introduced in \cite{Qiu20,IQ1}, serves as a key tool for our analysis. This function was initially developed to demonstrate the existence of $q$-deformations of stability conditions in \cite{IQ1}.
\begin{definition}[\Cref{def:gldim}]
The global dimension of a stability condition $\sigma=(Z,\cP)$ on a triangulated category $\mathcal D$ is defined as
\[
\mathrm{gldim}\,\sigma:=\mathrm{sup}\{\phi'-\phi\,|\,\mathrm{Hom}(\mathcal P(\phi),\mathcal P(\phi'))\ne0\}.
\]
\end{definition}
Our second result shows that the global dimension of stability conditions for $\p3$ constructed in \cite{BMS} is 3: 
\begin{theorem}[{\Cref{main2}}]
There is an injection
\[
\iota_3:\left\{(Z_{\alpha,\beta}^{a,b},\cA_{\alpha,\beta})[g]\in\Stab(\p3)\mid (\alpha,\beta,a,b)\in\mathfrak{B},g\in\gl2\right\} \hookrightarrow\mathrm{gldim}^{-1}(3).
\]
\end{theorem}
\subsection{Contractibility of stability condition spaces}
Exploring the homotopy type of stability condition spaces has become a fascinating area of research. It is hypothesized that if the a stability condition space is nonempty, it should be homotopy discrete, but this remains a distant goal in terms of proof. This area has been richly investigated, including contributions from \cite{B08,HMS,BM11,HKK,Qiu15, QW, DK16a,DK16b,DK19, BQS, FLZ,D22}. Recent research has concentrated on two particular open subspaces: the algebraic stability condition spaces and the geometric stability condition spaces. In many cases, the union of algebraic and geometric stability condition spaces form a connected component of the whole space.  Based on \cite{Li17}, we formulate a conjecture concerning the contractibility of a principal connected component $\stabp\subset\Stab(\p3)$ that contains image of $\Sigma_{\Psi}$ for $\p3$: 
\begin{conjecture}[\Cref{mian3}]
The principal connected compnent $\stabp\subset\Stab(\p3)$ is the union of geometric and algebraic stability conditions and is contractible.
\[
\stabp=\stabg\bigcup\staba.
\]
\end{conjecture}
For further exploration, a deeper understanding of full exceptional collections in $\cD^b(\p3)$ is required, as well as a complete description of $\stabg$. Similar to the non-constant global dimension of $\Stab^{\Geo}(\mathbb P^2)$ by \cite{FLLQ}, the global dimension of $\stabg$ is not expected always to be 3, which suggests that the subset contructed by \cite{BMS} may not be $\stabg$ by \Cref{main2}. The space $\Im\Sigma_{\Psi}$ from \Cref{main1} is expected to correspond to $\stabg$, with partial evidence from \Cref{partialgeo}. Following \cite{FLLQ}, it is also expected that $\stabp$ is contractible with respect to its global dimension.

\subsection{Contents}
This article is structured as follows:
In \Cref{sec:pre}, we review the concepts of Bridgeland stability conditions, the global dimension function, as well as slope and tilt stability.
In \Cref{sec:construct}, we construct a subset of the space of stability conditions for any projective threefold equipped with an ample polarization that satisfies a specific Bogomolov-Gieseker inequality.
In \Cref{sec.gldim}, we show that the global dimension of the subspace of stability conditions for $\mathbb{P}^3$ constructed in \cite{BMS} is 3, and we construct a family of stable vector bundles with respect to stability conditions on $\mathcal{D}^b(\mathbb{P}^3)$ in the large volume limit.
In \Cref{sec:algebr-stab-cond}, we formulate a conjecture concerning the contractibility of the principal connected component of $\mathrm{Stab}(\mathbb{P}^3)$.

\subsection{Notations}

\begin{longtable}{c|c}
\caption{Notations} \\
      $\mathcal D$   & a triangulated category \\
       $X$  & a smooth projective variety over $\mathbb C$\\
    $\mathcal D^b(X)$    & the bounded derived category of coherent sheaves on $X$\\
    $K(\mathcal D),K(X)$ & the Grothendieck group of $\mathcal D$, resp. $\mathcal D^b(X)$\\
    $K_{\num}(\mathcal D),K_{\num}(X)$ & the numerical Grothendieck group of $\mathcal D$, resp. $\mathcal D^b(X)$\\
    $\mathrm{Stab}(\mathcal D),\mathrm{Stab}(X)$ & the space of numerical Bridgeland stability conditions on $\mathcal D,\mathcal D^b(X)$\\
$\mathrm{Stab}_H(X)$ & the space of numerical Bridgeland stability conditions on $\mathcal D^b(X)$ \\
&with respect to $(\Lambda_H,\lambda_H)$\\
    $\mathrm{Stab}^{\mathrm{Geo}}(X)$ & the space of geometric numerical stability conditions on $\mathcal D^b(X)$\\
    $\mathrm{Stab}^{\mathrm{Alg}}(\mathcal X)$ & the space of algebraic numerical stability conditions on $\mathcal D^b(X)$\\
$\ch(E)$ & the Chern character of an object $E\in\cD^b(X)$ \\
    $\mathrm{NS}(X)$ & the N\'{e}ron-Severi group of $X$\\
    $\mathrm{Amp}_{\mathbb R}(X)$ & the ample cone inside $\mathrm{NS}_{\mathbb R}(X)$\\
    \label{notation}
\end{longtable}

\subsection*{Acknowledgement}
We appreciate the lecture on Bogomolov-Gieseker inequalities by Mao Sheng and valuable discussions with Bowen Liu. We are also thankful to Yuwei Fan and Takumi Otani for their helpful comments. Dongjian Wu is grateful to his supervisor Yu Qiu for his continuous support, encouragement and patience throughout his research. Nantao Zhang is grateful to his supervisor Will Donovan for encouragement and various suggestions.

\section{Preliminaries}
\label{sec:pre}
\subsection{Review: Bridgeland Stability conditions}
We begin by reviewing the concept of Bridgeland stability conditions on triangulated categories in \cite{B07}.  In our context, we assume that for the Grothendieck group $K(\mathcal D)$ of a triangulated category $\mathcal D$ is free of finite rank, i.e. $K(\mathcal D)\cong\mathbb Z^{\oplus n}$ for some $n$. 

A \emph{Bridgeland pre-stability condition} $\sigma=(Z, \mathcal P)$ on a triangulated category $\mathcal D$ is characterized by a group homomorphism $Z: K(\mathcal D)\to \mathbb C$, termed the \emph{central charge}, and a collection of full additive subcategories $\mathcal P(\phi)\subset \mathcal D$ for each $\phi\in \mathbb R$, called the \emph{slicing}. This pair is subject to the following axioms:

\begin{enumerate}
\item[(a)] if $0\ne E\in \mathcal P(\phi)$, then $Z(E) \in \mathbb R_{>0}\cdot{\rm exp}(\textbf i\pi\phi)$,
\item[(b)] for all $\phi\in \mathbb R$, $\mathcal P(\phi +1) = \mathcal P(\phi)[1]$,
\item[(c)] if $\phi_1>\phi_2$ and $A_i\in \mathcal P(\phi_i)\,(i=1,2)$, then ${\rm Hom}_{\mathcal D}(A_1, A_2) = 0$,
\item[(d)] for $0\neq E\in \mathcal D$, there is a finite sequence of real numbers
\[
\phi_1>\phi_2>\dots>\phi_m
\]
and a collection of triangles called \emph{Harder-Narasimhan filtration}
$$0 =
\xymatrix @C=5mm{
 E_0 \ar[rr]   &&  E_1 \ar[dl] \ar[rr] && E_2 \ar[dl]
 \ar[r] & \dots  \ar[r] & E_{m-1} \ar[rr] && E_m \ar[dl] \\
& A_1 \ar@{-->}[ul] && A_2 \ar@{-->}[ul] &&&& A_m \ar@{-->}[ul]
}
= E$$
with $A_i\in\mathcal P(\phi_i)$ for all $1\leq i\leq m$.
\end{enumerate}

Let $E$ be a nonzero object in $\mathcal D$ that admits a Harder-Narasimhan filtration as described in axiom (d). We associate two numbers with $E$: $\phi^+_{\sigma}(E):=\phi_1$ and $\phi^-_{\sigma}(E):=\phi_m$, where $\phi_1$ and $\phi_m$ are the phases from axiom (d). An object $E\in \mathcal P(\phi)$ for some $\phi\in \mathbb R$ is called semistable, and in such a case, $\phi = \phi^{\pm}_{\sigma}(E)$. Moreover, if $E$ is a simple object in $\mathcal P(\phi)$, it is said to be stable. We define $\mathcal P(I)$ for an interval $I$ in $\mathbb R$ as
 \[
 \mathcal P(I) = \{E\in \mathcal D\ |\ \phi^{\pm}_{\sigma}(E)\in I\}\cup\{0\}.
 \]
Consequently, for any $\phi\in \mathbb R$, both $\mathcal P[\phi, \infty)$ and $\mathcal P(\phi, \infty)$ are t-structures in $\mathcal D$.

Additionally, we restrict our attention to Bridgeland pre-stability conditions that fulfill the \emph{support property} (cf. \cite{KS}). A Bridgeland pre-stability condition $\sigma=(Z,\mathcal P)$ satisfies the support property (with repsect to $(\Lambda,\lambda)$) if 
 \begin{enumerate}
     \item $Z$ factors via a finite rank lattice $\Lambda$, i.e. $Z:K(\mathcal D)\xrightarrow{\lambda}\Lambda\to\mathbb C$, and 
     \item there exists a quadratic form $Q$ on $K_{\num}\otimes\mathbb R$ such that 
     \begin{enumerate}
         \item[(a)] $\Ker\,Z$ is negative definite with respect to $Q$, and 
         \item[(b)] every $\sigma$-semistable object $E\in\mathcal D$ satisfies $Q(\lambda(E))\ge0$. 
     \end{enumerate}
 \end{enumerate}
 A Bridgeland pre-stability condition that satisifes the support property is referred to as a \emph{Bridgeland stability condition}. If $\lambda$ factors through $K_{\num}(\mathcal D)$, $\sigma$ is called a \emph{numerical Bridgeland stability condition}. 
 
  The set of stability conditions with respect to $(\Lambda,\lambda)$ is denoted by $\mathrm{Stab}_{\Lambda}(\mathcal D)$. Unless otherwise specified, we will assume that all Bridgeland stability conditions are numerical. The set of numerical stability conditions on $\mathcal D$ is denoted by $\mathrm{Stab}(\mathcal D)$.

As described in \cite{B07}, $\mathrm{Stab}(\mathcal D)$ possesses a natural topology induced by the generalized metric
\[
d(\sigma_1,\sigma_2)=\sup\limits_{0\ne E\in\mathcal D}\left\{|\phi_{\sigma_2}^-(E)-\phi_{\sigma_1}^-|, |\phi^+_{\sigma_2}(E)-\phi^+_{\sigma_1}(E)|, \left|\mathrm{log}\frac{m_{\sigma_2}(E)}{m_{\sigma_1}(E)} \right| \right\}.
\]
\begin{theorem}\cite[Theorem 1.2]{B07}
\label{deformation}
The space of stability conditions $\mathrm{Stab}(\mathcal D)$ has the structure of a complex manifold, and the map
\[
\mathrm{Stab}(\mathcal D)\to\mathrm{Hom}_{\mathbb Z}(K_{\num}(\mathcal D),\mathbb C)
\]
that sends a stability condition to its central charge is a local isomorphism.
\end{theorem}

There is an alternative characterisation of Bridgeland stability conditions that relies on the notion of a t-structure on a triangulated category. Due to \cite{B07}, A \emph{stability function} on an abelian category $\mathcal C$ is a group homomorphism $Z : K(\mathcal C)\to \mathbb C$ that satisfies the following condition: for every nonzero objects $E\in\mathcal C$, the complex number $Z(E)$ lies in the semi-closed upper half plane
\[
\mathbb H_{+}=\{r\mathrm{exp}(i\pi\phi): r>0\ \text{and}\ 0<\phi\le1\}\subset\mathbb C.
\]

For every non-zero object $E$, we define the \emph{phase} as $\phi(E)=\frac{1}{\pi}\mathrm{arg}(Z([E]))\in(0,1]$. We say an object $E$ is $Z$-stable (resp. semistable) if $E\ne0$ and for every proper non-zero subobject $A$, we have $\phi(A)<\phi(E)$ (resp. $\phi(A)\le\phi(E)$). 

According to \cite[Proposition 5.3]{B07}, a stability condition $(Z,\mathcal P)$ on a triangulated category can be equivalently defined in terms of a pair $(\mathcal H,Z_{\mathcal H})$, where $\mathcal H$ is the heart of a bounded t-structure on $\mathcal D$ and $Z_{\mathcal H}$ is a stability function on $\mathcal H$ that satisfies the Harder-Narasimhan property. Furthermore, $(Z,\mathcal P)$ is a numerical Bridgeland stability condition if and only if $Z_{\mathcal H}$ factors via $K_{\num}(\mathcal D)$, and satisfies the support property for $Z_{\mathcal H}$-semistable objects. 

Let $\mathrm{Aut}(\mathcal D)$ denote the group of automorphisms of $\mathcal D$, $\mathrm{GL}^+_2(\mathbb R)$ be the group of elements in $\mathrm{GL}_2(\mathbb R)$ with positive determinant and let $\widetilde{\mathrm{GL}}_2^+(\mathbb R)$ be the universal cover of $\mathrm{GL}^+_2(\mathbb R)$. There exists a left action of $\mathrm{Aut}(\mathcal D)$ and a right action of the group $\widetilde{\mathrm{GL}}_2^+(\mathbb R)$ on $\mathrm{Stab}(\mathcal D)$ as described in \cite{B07}.  
For any $T\in\mathrm{Aut}(\mathcal D)$ and stability condition $\sigma=(Z,P)$, we define 
\[
T\sigma=(Z',P'),\quad Z'=ZT^{-1},\quad P'_{\phi}=TP_{\phi}.
\]

For any element $g=(T,f)\in\widetilde{\mathrm{GL}}_2^+(\mathbb R)$, we define
\[
\sigma[g]=(Z[g],P[g]), \quad Z[g]=T^{-1}Z,\quad P[g]_{\phi}=P_{f(\phi)}. 
\]
In particular, for $a+ib\in\mathbb C\subset\widetilde{\mathrm{GL}}_2^+(\mathbb R)$, we have
\[
Z[a+ib]=e^{-i\pi a+\pi b}Z, \quad P[a+ib]_{\phi}=P_{\phi+a}.
\]

\subsection{Global dimension function}
The concept of the global dimension function $\mathrm{gdim}$ on stability condition spaces is introduced in \cite{Qiu20,IQ1}. 

\begin{definition}[Global dimension]
\label{def:gldim}
The global dimension of a slicing $\cP$ on a triangulated category $\mathcal D$ is defined as
\[
\mathrm{gldim}(\mathcal P):=\mathrm{sup}\{\phi'-\phi\,|\,\mathrm{Hom}(\mathcal P(\phi),\mathcal P(\phi'))\ne0\}\in [0,+\infty].
\]
For a stability conditions $\sigma=(Z,\cP)$ on $\mathcal D$, the global dimension $\mathrm{gldim}\,\sigma$ is simply defined to be $\mathrm{gldim}\,\cP$.
\end{definition}

For an algebra $A$, let $\cP_A$ be the \emph{canonical slicing} on $\mathcal D^b(A)$, i.e. we have $\cP_A(0)=\mathrm{mod}\,A$ and $\cP_A(0,1)=\emptyset$. In this case, we have $\mathrm{gldim}\,\cP_A=\mathrm{gldim}\,A$. We observe the following facts (cf. \cite[Section 2]{IQ1}): for any triangulated category $\mathcal D$,
\begin{enumerate}
\item $\mathrm{gldim}$ is a continuous function on $\mathrm{Stab}\,\mathcal D$;
\item $\mathrm{gldim}$ is invariant under the $\mathbb C$-action and the action of $\mathrm{Aut}\,\mathcal D$.
 \end{enumerate}
Therefore, we can define a function
\[
\mathrm{gldim}: \mathrm{Aut}(\mathcal D)\backslash\mathrm{Stab}\,\mathcal D/\mathbb C\to[0,+\infty].
\]
The \emph{global dimension} $\mathrm{Gd}\,\mathcal D$ of $\mathcal D$ is given by
\[
\mathrm{Gd}\,\mathcal D:=\mathrm{inf}\,\mathrm{gldim}\,\mathrm{Stab}\,\mathcal D.
\]
Note that if $\mathrm{Stab}\,\mathcal D=\emptyset$, then $\mathrm{Gd}\,\mathcal D$ is not defined. See \cite{Kaw23} for a reference that demonstrates the nonexistence of Bridgeland stability conditions on $R$-linear triangulated categories of affine schemes with positive dimensions.

We say that a slicing $\cP$ (or $\sigma$) is $\mathrm{gldim}$-\emph{reachable} if there exist phases $\phi_1$ and $\phi_2$ such that 
\[
\mathrm{Hom}(\cP(\phi_1),\cP(\phi_2))\ne0, \quad \mathrm{gldim}\,\cP=\phi_2-\phi_1.
\]
We say $\mathcal D$ is $\mathrm{gldim}$-\emph{reachable} if there exists $\sigma$ such that $\mathrm{gldim}\,\sigma=\mathrm{Gd}\,\mathcal D$. By \cite{Qiu18,KOT}, we have the following results:
\begin{enumerate}
\item If $\mathcal D$ is the bounded derived category of the path algebra of an acyclic quiver $Q$, then $\mathcal D$ is gldim-reachable.
\item If $\mathcal D$ is the bounded derived category of the coherent sheaves on a smooth projective curve $X$ of genus $g$ (over $\mathbb C$), then $\mathrm{Gd}\,\mathcal D=1$ and 
\begin{itemize}
\item $\mathcal D$ is gldim-reachable if $g=0,1$;
\item $\mathcal D$ is not gldim-reachable if $g>1$.
\end{itemize}
\end{enumerate}

\subsection{Review: Slope stability and tilt stability}
In this section, we review the notion of slope stability and tilt stability for polarized projective threefolds to construct stability conditions \cite{BMT,M14}. Let $X$ be a smooth projective threefold with an ample polarization $H\in\mathrm{NS}(X)$. 

\subsubsection{Slope stability}
For $\beta\in\mathbb R$, we introduce a slope function $\mu_{\beta}$ for coherent sheaves on $X$ in the standard way: For $E\in\mathrm{Coh}(X)$, we set 
\begin{equation}
\mu_{\beta}=
\begin{cases}
+\infty, \quad & \mathrm{ch}_0^{\beta}(E)=0,\\
\frac{H^2\mathrm{ch}_1^{\beta}(E)}{H^3\mathrm{ch}_0^{\beta}(E)}, \quad & \text{else},
\end{cases}
\end{equation}
where $\mathrm{ch}^{\beta}(E):=e^{-\beta H}\mathrm{ch}(E)$ denotes the Chern character twisted by $B$. Explicitly:
\begin{equation*}
\begin{split}
\mathrm{ch}_0^{\beta}&=\mathrm{ch}_0;\\
\mathrm{ch}_1^{\beta}&=\mathrm{ch}_1-\beta H\mathrm{ch}_0;\\
\mathrm{ch}_2^{\beta}&=\mathrm{ch}_2-\beta H\mathrm{ch}_1+\frac{(\beta H)^2}{2}\mathrm{ch}_0;\\
\mathrm{ch}_3^{\beta}&=\mathrm{ch}_3-\beta H\mathrm{ch}_2+\frac{(\beta H)^2}{2}\mathrm{ch}_1-\frac{(\beta H)^3}{6}\mathrm{ch}_0.
\end{split}
\end{equation*}
Note that $\mu_{\beta}(\cE)=\mu(\cE)-\beta$, where $\mu(\cE)=\mu_0(\cE)$ is the classical slope stability condition function. A coherent sheaf $E$ is slope-(semi)stable (or $\mu_{\beta}$-(semi)stable) if for all proper subsheaves $F\hookrightarrow E$, we have 
\[
\mu_{\beta}(F)<(\le) \mu_{\beta}(E/F).
\]

According to the existence of Harder-Narasimhan filtrations with respect to slope-stability, there exists a \emph{torsion pair} $(\mathcal T_{\beta},\mathcal F_{\beta})$ defined as follows:
\begin{equation*}
\begin{split}
\mathcal T_{\beta}&=\{E\in\mathrm{Coh}(X): \text{ any quotient } E\twoheadrightarrow G \text{ satisfies } \mu_{\beta}(G)>0\};\\
\mathcal F_{\beta}&=\{E\in\mathrm{Coh}(X): \text{ any  subsheaf } F\hookrightarrow E\text{ satisfies } \mu_{\beta}(F)\le0\}.
\end{split}
\end{equation*}
Equivalently, $\mathcal T_{\beta}$ and $\mathcal F_{\beta}$ are the extension-closed subcategories of $\mathrm{Coh}(X)$ generated by slope-stable sheaves of positive or non-positive slope, respectively. Tilt the category $\mathrm{Coh}(X)$ with respect to this torsion pair and denote the resulting heart by $\mathrm{Coh}^{\beta}(X)=\langle\mathcal T_{\beta},\mathcal F_{\beta}[1]\rangle$. According to the general theory of torsion pairs and tilting \cite{HRS}, $\mathrm{Coh}^{\beta}(X)$ is the heart of a bounded t-structure on $\mathcal D^b(X)$.

We now state a property of objects in $\Cohb(X)$ as follows:

\begin{lemma}[{\cite[Lemma 3.2.1]{BMT}}]
\label{nuch1}
For any non-zero object $E\in\mathrm{Coh}^{\beta}(X)$, one of the following conditions holds: 
\begin{enumerate}
\item $H^2\mathrm{ch}^{\beta}_1(E)>0$;
\item $H^2\mathrm{ch}_1^{\beta}(E)=0$ and $\Im\, Z_{\alpha,\beta}(E)>0$;
\item $H^2\mathrm{ch}_1^{\beta}(E)=\Im\, Z_{\alpha,\beta}(E)=0$ and $-\Re\, Z_{\alpha,\beta}(E)>0$.
\end{enumerate}
\end{lemma}

\subsubsection{Tilt stability}

According to \cite{BMT}, we have the following slope function on $\mathrm{Coh}^{\beta}(X)$: For $E\in\mathrm{Coh}^{\beta}(X)$, define
\begin{equation*}
\nu_{\alpha,\beta}=
\begin{cases}
+\infty, \quad & \text{ if } (\alpha H)^2\mathrm{ch}_1^{\beta}(E)=0;\\
\frac{\alpha H\mathrm{ch}_2^{\beta}(E)-\frac{1}{2}(\alpha H)^3\mathrm{ch}_0^{\beta}(E)}{\alpha^2H^2\mathrm{ch}_1^{\beta}(E)}, \quad &\text{ else}.
\end{cases}
\end{equation*}

An object $E\in\Cohb(X)$ is \emph{tilt-(semi)stable} if, for all non-trivial subobjects $F\hookrightarrow E$, we have 
\[
\nu_{\alpha,\beta}<(\le)\nu_{\alpha,\beta}(E/F).
\] 

Consider the following full subcategories of $\mathrm{Coh}^{\beta}(X)$
\begin{equation*}
\begin{split}
\mathcal T'_{\alpha,\beta}&=\{E\in\mathrm{Coh}^{\beta}(X): \text{ any quotient } E\twoheadrightarrow G \text{ satisfies } \nu_{\alpha,\beta}(G)>0\};\\
\mathcal F'_{\alpha,\beta}&=\{E\in\mathrm{Coh}^{\beta}(X): \text{ any  subsheaf } F\hookrightarrow E\text{ satisfies } \nu_{\alpha,\beta}(F)\le0\}.
\end{split}
\end{equation*}
They also form a torsion pair. Then the tilt of $\mathrm{Coh}^{\beta}(X)$ with respect to $(\mathcal T'_{\alpha,\beta},\mathcal F'_{\alpha,\beta})$ provides us with the abelian category 
\[
\mathcal A_{\alpha,\beta}=\langle\mathcal T'_{\alpha,\beta},\mathcal F'_{\alpha,\beta}[1] \rangle.
\]


\section{Construction of stability conditions on polarized projective threefolds}
\label{sec:construct}

In this section, we start by revisiting some Bogomolov-Gieseker inequalities for projective threefolds. Then, we construct an explicit subset of the space of the space of stability conditions for any projective threefold with an ample polarization that satisfy the generalized Bogomolov-Gieseker inequality. This is achieved by utilizing the function $\Psi_X$ in \Cref{Psi} to refine the result in \cite{BMS}, based on the insights from \cite{BMT,BMS,FLZ,D22}.

In the following sections, the term ``a polarized threefold $(X,H)$'' refers to a smooth projective threefold equipped with an ample polarization $H$ that obeys the \emph{generalized Bogomolov-Gieseker inequality} as stated in \Cref{Bog}. We consider the lattice $\Lambda_H$ generated by vectors of the form 
\[
(H^3\mathrm{ch}_0(E),H^2\mathrm{ch}_1(E),H\mathrm{ch}_2(E),\mathrm{ch}_3(E))\in\mathbb Q^4
\] 
along with the natural map $\lambda_H:K(X)\to\Lambda_H$. According to \Cref{deformation}, the space $\Stab_H(X)$ of stability conditions on $\mathcal D^b(X)$ with respect to $(\Lambda_H,\lambda_H)$ is a four-dimensional complex manifold such that the map
\[
Z:\Stab_H(X)\to \mathrm{Hom}(\Lambda_H,\mathbb C), \quad (Z,\cP)\mapsto Z
\]
is a local isomorphism.  

\subsection{Reivew: Bogomolov-Gieseker inequalities}
We first review the classical and generalized Bogomolov-Gieseker inequalities for tilt stable objects on smooth projective threefolds with ample polarizations as follows: 

\begin{lemma}[\cite{BMT,BMS}]
\label{Bog}
Suppose that $X$ is a smooth projective threefold with an ample polarization $H$. Then any $\nu_{\alpha,\beta}$-semistable object $E\in\mathrm{Coh}^{\beta}(X)$ satisfies the {(classical) Bogomolov-Gieseker inequality}:
\[
(H\mathrm{ch}_1^{\beta}(E))^2-2H^2\mathrm{ch}_0^{\beta}(E)H\mathrm{ch}_2^{\beta}(E)\ge0.
\]
Moreover, the {generalized Bogomolov-Gieseker inequality} is expressed as
\[
\mathrm{ch}_3^{\beta}(E)\le\frac{\alpha^2 H^2}{6}\mathrm{ch}_1^{\beta}(E), \emph{\text{ for }}  \nu_{\alpha,\beta}(E)=0,
\]
which is equivalent to,
\[
Q_{\alpha^2}^{\beta}:=\alpha^2\overline{\Delta}_{H}(E)+\overline{\nabla}_{H}^{\beta}(E)\ge0,
\]
where
\[\overline{\Delta}_{H}(E) = (H^{2}\ch_{1}(E))^{2} - 2 H^{3}\ch_{0}(E) H\ch_{2}(E),\] 
and 
\[\overline{\nabla}_{H}^{\beta}(E)= 4(H \ch_{2}^{\beta}(E))^{2} - 6 H^{2} \ch_{1}^{\beta}(E) \ch_{3}^{\beta}(E).\]

\end{lemma}

Given a pair $(\alpha,\beta)\in\mathbb R_{>0}\times\mathbb R$, we consider the central charge defined as follows:
\begin{equation*}
\begin{split}
Z_{\alpha,\beta}(E)&=-\int_Xe^{-\beta H-i\alpha H}\mathrm{ch}(E)\\
&=\left(-\mathrm{ch}_3^{\beta}(E)+\frac{(\alpha H)^2}{2}\mathrm{ch}_1^{\beta}(E)\right)+i\left(\alpha H\mathrm{ch}_2^{\beta}(E)-\frac{(\alpha H)^3}{6}\mathrm{ch}_0^{\beta}(E)\right).
\end{split}
\end{equation*}

According to \cite{BMT}, the criteria for associating a stability condtion with $Z_{\alpha,\beta}$ is as follows:

\begin{theorem}[{\cite[Corollary 5.2.4]{BMT}}]
\label{standard}
Let $(\alpha,\beta)\in\mathbb R_{>0}\times\mathbb R$. The pair $(Z_{\alpha,\beta},\mathcal A_{\alpha,\beta})$ forms a stability condition on $\mathcal D^b(X)$ if and only if, for any $\nu_{\alpha,\beta}$-semistable object $E\in\mathrm{Coh}^{\beta}(X)$ satisfying 
\[
\frac{(\alpha H)^3}{6}\mathrm{ch}_0^{\beta}(E)=\alpha H\mathrm{ch}_2^{\beta}(E),
\]
we have 
\[
\mathrm{ch}_3^{\beta}(E)<\frac{(\alpha H)^2}{2}\mathrm{ch}_1^{\beta}(E).
\]
\end{theorem}

\begin{remark}
According to \cite{M14} and \cite{BMS}, the generalized Bogomolov-Gieseker inequality holds for $\p3$ and polarized abelian threefolds. Let $(X,H)$ be a polarized abelian threefold. Then by \Cref{standard}, we have that for any pair $(\alpha,\beta)\in\mathbb R_{>0}\times\mathbb R$, $(Z_{\alpha,\beta},\mathcal A_{\alpha,\beta})$ defines a stability condition on $\mathcal D^b(\p3)$ and $\mathcal D^b(X)$, respectively . We can thus define
\[
\stabgc:=\{\sigma_{\alpha,\beta}=(Z_{\alpha,\beta},\mathcal A_{\alpha,\beta}):(\alpha,\beta)\in\mathbb R_{>0}\times\mathbb R\}\subset\mathrm{Stab}(\p3),
\]
and similarly, we can define $\Stab^{\Geo}_{\Delta}(X)$. Note that for any $\sigma_{\alpha,\beta}=(Z_{\alpha,\beta},\cP_{\alpha,\beta})\in\stabgc$, we have $Z_{\alpha,\beta}(\cO_p)=1$ and $\cO(p)\in\cP_{\alpha,\beta}(1)$. In \Cref{sec.gldim}, we can derive that $\stabgc$ is of dimemsion 3 as a $\gldim$-reachable open subset of $\mathrm{Stab}(\p3)$.

\end{remark}

\subsection{Description of geometric stability conditions}
In this section, we construct an explicit open subset of the space of the space of stability conditions $\stabh(X)$ for any polarized threefold $(X,H)$ by employing the function $\Psi_X$ in \Cref{Psi} to refine the result in \cite{BMS}, based on insights from \cite{BMT,BMS,FLZ,D22}.

\begin{definition}
A stability condition $\sigma$ on $\mathcal D^b(X)$ is called \emph{geometric} (with respect to $X$) if for each point $p\in X$, the skyscraper sheaves are $\sigma$-stable.
\end{definition}

\begin{remark}
Let $\sigma$ be a geometric numerical stability condition on $\mathcal D^b(X)$. Due to \cite[Proposition 2.9]{FLZ}, all skyscraper sheaves are of the same phase. The subspace $\mathrm{Stab}^{\mathrm{Geo}}(X)$ is open in $\mathrm{Stab}(X)$, so each connected component of $\mathrm{Stab}^{\mathrm{Geo}}(X)$ is a complex submanifold of some connected component of $\mathrm{Stab}(X)$ \cite[Proposition 9.4]{B08}. By \cite[Proposition 2.1]{MP}, skyscraper sheaves are stable for all $\sigma\in\stabgc$, i.e. $\stabgc\subset\stabg$. This fact also applies to polarized abelian varieties. 
\end{remark}


\begin{lemma}
  \label{lem:StabP3-2}
Let $\sigma = (Z, \cP)$ be a geometric stability condition on $D^b(X)$ with $Z(\mathcal O_x) = -1$ and $\mathcal O_x \in \cP(1)$ and $\dim X=n$. Then the following are true:
\begin{enumerate}
    \item If $E \in \cP(\linterval{0}{1})$, then $H^i(E) \neq 0$ unless $i \in \{-n+1, \dots, -1,0\}$ and $H^{-n+1}(E)$ is torsion-free.
    \item If $E \in \cP(1)$ is stable, then either $E$ is $\mathcal O_x$ or isomorphic to a complex of locally free sheaves $0 \to E^{-n + 1} \to \dots \to E^{-1} \to 0$ supported on $[-n + 1, -1]$.
    \item If $E \in \Coh(X)$, then $E \in \cP(\linterval{-n+1}{1})$. Moreover, if $E$ is a sheaf with support dimension $\leq k$, then $E \in \cP(\linterval{1-i}{1})$. Here we abuse the notation by setting $\cP(\linterval{1}{1}) = \cP(1)$.
\end{enumerate}
\end{lemma}

\begin{proof}
This is a straightforward extension of \cite[Lemma 10.1]{B08}, and the proof proceeds in a similar fashion.
\end{proof}

\begin{lemma}
\label{reduce}
Let $(X,H)$ be a polarized threefold. Then, for \(\sigma \in \stabgh(X)\), we have \(\sigma = (Z, \cA)[g]\)  for some \(g \in \widetilde{\GL}_{2}^{+}(\RR)\) where
  \[Z = Z^{a, b, c}_{d, \beta} := - \ch_{3}^{\beta} + bH \ch_{2}^{\beta} + aH^{2} \ch_{1}^{\beta} + cH^{3} \ch_{0}^{\beta} + i(H \ch_{2}^{\beta} - dH^{3} \ch_{0}^{\beta})\]
   for \(\alpha > 0\). Furthermore, if $d>0$, we can express $Z$ as:
\[Z^{a,b}_{\alpha, \beta}  := - \ch_{3}^{\beta} + bH \ch_{2}^{\beta} + aH^{2} \ch_{1}^{\beta} + i(H \ch_{2}^{\beta} - \frac{\alpha^{2}}{2}H^{3} \ch_{0}^{\beta})\]
  where \(\alpha < 0\).
\end{lemma}

\begin{proof}
By the universal property of Chern character, we have
  \[Z = a_{1} \ch_{3} + a_{2} H \ch_{2} + a_{3} H^{2} \ch_{1} + a_{4} H^{3} \ch_{0} +
  i(b_{1} \ch_{3} + b_{2} H \ch_{2} + b_{3} H^{2} \ch_{1} + b_{4} H^{3} \ch_{0})\]
  Since \(\sigma\) is geometric, we have all \(\cO_{x}\) have the same central charge and phase \cite[Proposition 2.9]{FLZ}. Moreover, \(Z(\cO_{x}) \neq 0\), so by action of \(\widetilde{\GL}_{2}^{+}(\RR)\), we may assume that \(Z(\cO_{x}) = -1\) and \(\cO_{x} \in P(1)\), i.e. \(a_{1} = -1\) and \(b_{1} = 0\). Therefore,
  \[Z = - \ch_{3}^{\beta} + a_{2} H \ch_{2}^{\beta} + a_{3} H^{2} \ch_{1}^{\beta} + a_{4} H^{3} \ch_{0}^{\beta} +
    i(b_{2} H \ch_{2}^{\beta} + b_{3} H^{2} \ch_{1}^{\beta} + b_{4} H^{3} \ch^{\beta}_{0})\]
  By~\Cref{lem:StabP3-2}, we have \(\cO_{C} \in \cP(\linterval{0}{1})\) for every curve \(C \subset \PP^{3}\). Hence \(b_{2} H \ch_{2}^{\beta}(\cO_{S}) = b_{2} H C \geq 0\) and thus \(b_{2} \geq 0\). However due to the deformation property and openness of geometric stability condition, we have \(b_{2} > 0\). By the action of \(\RR \subset \widetilde{\GL}_{2}^{+}(\RR)\) scaling the \(y\)-axis, we can set \(b_{2} = 1\). By changing from \(\ch_{i}\) to \(\ch_{i}^{\beta}\), we may assume that \(b_{3} = 0\).

Assume now that $b>0$. We may write $b=\frac{\alpha^2}{2}$ for some $\alpha>0$.  By using the shearing action of \(\RR \subset \widetilde{\GL}_{2}^{+}(\RR)\), we can eliminate either \(b\) or \(c\). We choose to eliminate \(c\) and express \(Z\) as required.
\end{proof}

From now on, we focus on a particular subspace of stability condition space:
\begin{definition}
\label{stabghd}
Let $X$ be a projective threefold with an ample polarization $H$. We define the subspace of stability condition space \[\Stab^{\mathrm{Geo}\dagger}_{H}(X) :=\left\{\sigma=(Z_{\alpha, \beta}^{a,b}, \cA_{\alpha,\beta})[g]  \in \Stab_H(X) \mid g \in \widetilde{\GL}_{2}^{+}(\RR), \alpha> 0 \right\},\] where $Z_{\alpha,\beta}^{a,b}$ is given as in \Cref{reduce}.
\end{definition}
Note that $\Stab^{\Geo\dagger}(X)\subset\Stab^{\Geo}(X)$ by \cite[Proposition 2.1]{MP}. Motivated by the Le Potier function as discussed in \cite{FLZ,D22}, we define the following function:

\begin{definition}
\label{Psi}
Let $X$ be a projective threefold with an ample polarization $H$. We define the function twisted by $\beta H$, $\Psi_{X,\nu}:\mathbb R_{>0}\times\mathbb R^2\to\mathbb R\cup\{-\infty\}$, with respect to a parameter $\nu\in\mathbb R$, as follows:
\[
\Psi_{X,\nu}(\alpha, \beta, b):=  \limsup_{\mu \to \nu} \left\{\frac{\ch_{3}^{\beta}(F) - bH \ch_{2}^{\beta}(F)}{H^{2} \ch_{1}^{\beta}(F)}:
{\begin{aligned}
&F\in\Cohb(X) \text{ is } \nu_{\alpha, \beta}\text{-semistable}\\ 
&\text{ with } \nu_{\alpha, \beta}(F) = \mu
\end{aligned}}
\right\},
\]
if the limit exists, and $\Psi_{X,\nu}(\alpha,\beta,b):=-\infty$ otherwise. In particular, when $\nu=0$, we denote by $\Psi_{X}:=\Psi_{X,0}$ for simplicity.
\end{definition}
Let $F\in\Cohb(X)$ be $\nu_{\alpha,\beta}$-semistable. By \Cref{nuch1}, $H^2\mathrm{ch}_1(X)\ge0$. Note that, if $H^2\mathrm{ch}_1(X)=0$, $\nu_{\alpha,\beta}=+\infty$, which ensures that $\Psi_{X,\nu}$ is well-defined.

\begin{lemma}
\label{bound}
Let $(X,H)$ be a polarized threefold. Then, for any triple $(\alpha,\beta,b)\in\mathbb R_{>0}\times\mathbb R^2$,  there is quadratic function $\Xi$ of $\nu\in\mathbb R$ such that 
  \[\Psi_{X,\nu}(\alpha, \beta, b) \leq \Xi_{\alpha,\beta,b}(\nu).\]
In particular, if $\nu=0$, we have 
\[
\Psi_X(\alpha,\beta,b)\le\frac{\alpha^2}{6}+\frac{\alpha}{2}|b|.
\]

\end{lemma}
\begin{proof}
  By the generalized Bogomolov inequality on \(\Cohb(X)\) and \(\ch_{1}^{\beta}(E) \geq 0\) for \(\nu_{\alpha, \beta}\)-semistable \(E \in\Cohb(X)\) with \(\nu = \nu_{\alpha, \beta}(E) \neq \infty\), we obtain
  \[\frac{\ch_{3}^{\beta}(E) - b H\ch_{2}^{\beta}(E)}{H^{2} \ch_{1}^{\beta}(E)} \leq \frac{\alpha^{2}}{6} + \frac{(4 \nu - 6b)}{6} \frac{H \ch_{2}^{\beta}(E)}{H^{2} \ch_{1}^{\beta}(E)}\]
Furthermore, the Bogomolov inequality yields
  \[\frac{4}{\alpha^{2}} \frac{(H \ch_{2}^{\beta}(E))^{2}}{(H^{2} \ch_{1}(E))^{2}} - \frac{4 \nu}{\alpha^{2}} \frac{H \ch_{2}^{\beta}(E)}{H \ch_{1}^{\beta}(E)} -1 \leq 0\]
  which simplifies to
  \[\frac{\nu - \sqrt{\nu^{2} + \alpha^{2}}}{2} \leq \frac{H \ch^{\beta}_{2}(E)}{H^{2} \ch^{\beta}_{1}(E)} \leq \frac{\nu + \sqrt{\nu^{2} + \alpha^{2}}}{2}\]
  Consequently, 
  \[\Psi_{X,\nu}(\alpha, \beta, b) \leq \frac{1}{6} \alpha^{2} + \abs{\frac{2}{3} \nu - b}(\nu + \frac{\alpha}{2}) \leq \frac{1}{6} \alpha^{2} + \frac{1}{2}(\frac{2}{3} \nu - b)^{2} + \frac{1}{2}(\nu + \frac{\alpha}{2})^{2}=: \Xi_{\alpha,\beta,b}(\nu)\]
 Moreover, if we set \(\nu = 0\), we obtain
  \[\Psi_X(\alpha, \beta, b) \leq \frac{\alpha^{2}}{6} + \frac{\alpha}{2} \abs{b}.\]
\end{proof}

\begin{proposition}
\label{abelianfold}
Let $(X,H)$ be a polarized abelian threefold. Then, for any triple $(\alpha,\beta,b)\in\mathbb R_{>0}\times\mathbb R^2$, we have
\[
\Psi_X(\alpha,\beta,b)=\frac{\alpha^2}{6}+\frac{\alpha}{2}|b|.
\]
\end{proposition}
\begin{proof}
Consider the simple semi-homogeneous vector bundles $E_{p/q}$, as defined by Mukai, with slope $\frac{p}{q}$ and Chern character 
\[
\mathrm{ch}(E_{p/q})=\mathrm{ch}_0(E_{p/q})e^{\frac{p}{q}H}.
\] 
By \cite{Pol14,FLZ}, these bundles $E_{p/q}$ are stable with respect to any stability condition $\sigma\in\Stab_H(X)$. Find two rational sequences $\{{p_n/q_n}\}$ and $\{{p_n'/q_n'}\}$ that converge to $\alpha+\beta$ and $-\alpha+\beta$, respectively, as $n\to\infty$. Note that if $\frac{p}{q}>\beta$, $E_{p/q}\in\Cohb(X)$ and $\frac{p}{q}\le\beta$, $E_{p/q}[1]\in\Cohb(X)$. Thus, as $n\to\infty$, 
\[
\frac{\ch_{3}^{\beta}(E_{p_n/q_n}) - bH \ch_{2}^{\beta}(E_{p_n/q_n})}{H^{2} \ch_{1}^{\beta}(E_{p_n/q_n})}\to\frac{\alpha^2}{6}-\frac{\alpha}{2}b,\quad \frac{\ch_{3}^{\beta}(E_{p'_n/q'_n}[1]) - bH \ch_{2}^{\beta}(E_{p'_n/q'_n}[1])}{H^{2} \ch_{1}^{\beta}(E_{p'_n/q'_n}[1])}\to\frac{\alpha^2}{6}+\frac{\alpha}{2}b.
\]
In addition, by \cite[Corollary 3.11]{BMS}, $E_{p/q}$ is $\nu_{\alpha,\beta}$-stable. Hence, 
\[
\Psi_X(\alpha,\beta,b)\ge\frac{\alpha^2}{6}+\frac{\alpha}{2}|b|.
\]
Combining with \Cref{bound}, we obtain the required result.
\end{proof}

Using a similar strategy to exam twisted bundles $\mathcal O(d)$ on $\p3$, we can derive the following result:
\begin{proposition}
For any triple $(\alpha,\beta,b)\in\mathbb Z_{>0}\times\mathbb Z\times\mathbb R$, we have 
\[
\Psi_{\p3}(\alpha,\beta,b)=\frac{\alpha^2}{6}+\frac{\alpha}{2}|b|.
\]
\end{proposition}

Let us denote 
\[
\mathfrak{B}_{\Psi}:=\left\{(\alpha,\beta,a,b)\in\mathbb R^4\mid \alpha>0,a>\max\{\frac{\alpha^2}{6},\Psi_X(\alpha,\beta,b)\}\right\},
\]
and 
\[
\mathfrak{B}^{\ast}_{\Psi}:=\left\{(\alpha,\beta,a,b)\in\mathbb R^4\mid \alpha>0,a>\Psi_X(\alpha,\beta,b)\}\right\}.
\]
We proceed to state the main result of this section and the proof is based on the idea in \cite{BMT,BMS,D22}:
\begin{theorem}
\label{main1}
Let $(X,H)$ be a polarized threefold. Then, there is a continuous open embedding 
\begin{equation*}
\begin{split}
\Sigma_{\Psi}: \gl2\times\mathfrak{B}_{\Psi}&\to\stabgh(X)\\
(g,(\alpha,\beta,a,b))&\mapsto(Z_{\alpha, \beta}^{a, b}[g], \cA_{\alpha, \beta}[g]).
\end{split}
\end{equation*}
Furthermore, if $\Psi_{X}(\alpha,\beta,b)\ge\frac{\alpha^2}{6}$ for all $(\alpha,\beta,b)\in\mathbb R_{>0}\times\mathbb R^2$, we have 
\[
\Stab^{\mathrm{Geo}\dagger}_{H}(X)\cong\gl2 \times\mathfrak{B}_{\Psi}^{\ast}.
\]
\end{theorem}

\begin{proof}
Let $(\alpha,\beta,a,b)\in\mathfrak{B}_{\Psi}$. We claim that 
\begin{enumerate}
\item For any $0\ne E\in\mathcal A_{\alpha,\beta}$, it holds that $Z_{\alpha,\beta}^{a,b}(E)\in\mathbb H_+$;
\item Every $E\in\cA_{\alpha,\beta}$ admits a Harder-Narasimhan filtration with respect to $\sigma_{\alpha,\beta}^{a,b}$;
\item $(Z^{a,b}_{\alpha,\beta})$ satisfies the support property.
\end{enumerate}

Claims (1) and (2) can be deduced using the similar arguments as in \cite[Corollary 5.2.4]{BMT} and \cite[Theorem 8.6]{BMT}. The only point of distinction to note is that, for $F\in\Cohb(X)$ with $\nu_{\alpha,\beta}(F)=0$, we have 
\begin{equation}
\begin{split}
Z_{\alpha,\beta}^{a,b}(F[1])&=\ch_{3}^{\beta}(F) - bH \ch_{2}^{\beta}(F) - aH^{2} \ch_{1}^{\beta}(F)\\
&\le\Psi_X(\alpha,\beta,b)H^2\mathrm{ch}_1^{\beta}(F)-aH^2\ch_1^{\beta}(F) \\
&<0.
\end{split}
\end{equation}
Here, $H^2\ch_1^{\beta}(F)>0$. For (3), consider the following quadratic form
  \[S_{\delta} = \delta^{-1}(H\ch^{\beta}_{2} - \frac{\alpha^{2}}{2}H^{3}ch^{\beta}_{0})^{2} - H^{2}\ch^{\beta}_{1}(\ch^{\beta}_{3} - b H\ch^{\beta}_{2}  - (a - \delta) H^{2}\ch^{\beta}_{1}).\]

 We assert that there exists \(\delta > 0\), such that \(S_{\delta}(E) \geq 0\) for all \(\nu_{\alpha, \beta}\)-semistable \(E\). If $H^2\ch_1^{\beta}(E)=0$, it is evident. Thus, we only need to consider the case where $H^2\ch_1^{\beta}(E)>0$. Let $\nu:=\nu_{\alpha,\beta}(E)$. Then, 
 \[S_{\delta}(E) = (H \ch_{1}^{\beta}(E))^{2}(\delta^{-1} \nu^{2} + a - \delta - \frac{\ch_{3}^{\beta}(E) - b H \ch_{2}^{\beta}(E)}{H^{2} \ch_{1}(E)}).\]

If $H\ch_2^{\beta}(E)-\frac{\alpha^2}{2}H^3\ch_0^{\beta}(E)=0$, i.e. $\nu=0$, we have 
\begin{equation}
\begin{split}
S_{\delta}(E)&=(H\ch_2^{\beta}(E))^2(a-\delta-\frac{\ch_3^{\beta}(E)-bH\ch_2^{\beta}(E)}{H^2\ch_1^{\beta}(E)})\\
&\ge(H\ch_2^{\beta}(E))^2(a-\delta-\Psi_X(\alpha,\beta,b))\\
&>0,
\end{split}
\end{equation}
for $0<\delta<a-\Psi_X(\alpha,\beta,b)$. If $H\ch_2^{\beta}(E)-\frac{\alpha^2}{2}H^3\ch_0^{\beta}(E)\ne0$, then by \Cref{bound}, we have 
\[
 \frac{\ch_{3}^{\beta}(E) - b H \ch_{2}^{\beta}(E)}{H^{2} \ch_{1}^{\beta}(E)}\le\Xi_{\alpha,\beta,b}(\nu).
\]
For a sufficiently small \(\delta > 0\), this yields
  \[\frac{\ch_{3}^{\beta}(E) - bH \ch_{2}^{\beta}(E)}{H^{2} \ch_{1}^{\beta}(E)} < \delta^{-1} \nu^{2} + a - \delta.\]


Therefore, $S_{\delta}(E)\ge0$ for all $\nu_{\alpha,\beta}$-semistable $E$.  When Restricted to $\mathrm{Ker}\,Z_{\alpha,\beta}^{a,b}$, we obtain that $S_{\delta}(E)=-\delta(H^2\ch_1^{\beta}(E))^2$, which is negative semi-definite. Now we fix \(\delta\) and claim that there exists \(\epsilon > 0\), such that
  \[S_{\delta, \epsilon} = S_{\delta} + \epsilon Q_{\frac{1}{2}(\alpha^{2} + 6a)}^{\beta}(E)\]
  fulfils the support condition. By combining \Cref{thm:const} with the preceding discussion, we conclude that $S_{\delta,\epsilon}(E)\ge0$ for all $\sigma_{\alpha,\beta}^{a,b}$-semistable $E$. For $0\ne E\in\mathrm{Ker}\,Z^{a,b}_{\alpha,\beta}$, we have 
\begin{equation}
\begin{split}
\Re(Z_{\alpha,\beta}^{a,b})&=-\ch_3^{\beta}(E)+bH\ch_2^{\beta}(E)+aH^2\ch_1^{\beta}(E)=0;\\
\Im(Z_{\alpha,\beta}^{a,b})&=H\ch_2^{\beta}(E)-\frac{\alpha^2}{2}H^3\ch_0^{\beta}(E)=0.\\
\end{split}
\end{equation}
Hence,
 \[S_{\delta, \epsilon}(E) = (- \delta + \epsilon(\frac{1}{2} \alpha^{2} - 3a))(H^2\ch_{1}^{\beta}(E))^{2} + \epsilon(2 - \frac{12a}{\alpha^{2}})(H\ch_{2}^{\beta}(E))^{2} - 6 \epsilon bH^2\ch_{1}^{\beta}(E)H\ch_{2}^{\beta}(E).\]


If $H^2\ch^{\beta}_1(E)=0$, the claim is clear since $H\ch_2^{\beta}(E)\ne0$. On the other hand, we assert that there is an $\epsilon>0$ such that $S_{\delta,\epsilon}$ is negative definite on $\mathrm{Ker}\,Z_{\alpha,\beta}^{a,b}\backslash A$, where $A:=\{E\in K_{\num}\mid H^2\ch_1^{\beta}(E)=0\}$. 

We fix a norm on $K_{\num}(X)$ and denote $U:=\{E\in \mathrm{Ker}\,Z_{\alpha,\beta}^{a,b}(X)\mid\Vert E\Vert=1\}$. It suffices to find some $\epsilon>0$ such that $S_{\delta,\epsilon}$ is negative definite on $U\backslash A_1$, where $A_1:=\{E\in A\mid \Vert E\Vert=1\}$. Suppse, for the sake of contradiction, that for any $\epsilon>0$, there is $E\in U\backslash A_1$, such that 
\[
P(E):=\frac{Q_{\frac{1}{2}(\alpha^2+6a)}^{\beta}}{\delta(H^2\ch_1^{\beta}(E))^2}\ge\frac{1}{\epsilon}.
\]
This implies that $P$ is not bounded above on $U\backslash A_1$. However, since $Q_{\frac{1}{2}(\alpha^2+6a)}^{\beta}$ is negative on the closed set $A_1$, it must be negative definite on some open neighbourhood $V_1$ of $A_1$. Then, $P$ is negative on $V$. As $U\backslash V$ is compact, $P$ is bounded above $U\backslash A_1$, which contradicts our initial assumption. In conclude, this establishes the existence of the required $\epsilon$.

Following the same argument as in \cite[Proposition 8.10]{BMS}, we ultimately conclude that $\Sigma_{\Psi}$ is a continuois open embedding. 

For the moreover part, given that \(a > \frac{\alpha^2}{6}\), the preceding discussion shows that the support property is satisfied for $(Z_{\alpha,\beta}^{a,b},\cA_{\alpha,\beta})$ if and only if $(\alpha,\beta,a,b)\in\mathfrak{B}^{\ast}_{\Psi}$, which implies that 
\[
\stabh^{\Geo\dagger}\cong\gl2\times\mathfrak{B}^{\ast}_{\Psi}.
\]
\end{proof}

Based on some insights from \cite{FLZ}, we can also characterize the space of stability conditions on polarized abelian threefolds by applying \Cref{main1}:  

\begin{corollary}[{\cite[Theorem 4.8]{FLZ}}]
Let $(X,H)$ be a polarized abelian threefold. Then
\[
\Stab_H(X)\cong\gl2\times\mathfrak{B}_{\Psi}^{\ast}.
\]
In particular, $\Stab_H(X)$ is contractible.
\end{corollary}
\begin{proof}
According to \cite{FLZ}, we have $\Stab_H(X) = \Stab_H^{\mathrm{Geo}}(X)$. By deforming the family of simple semi-homogeneous vector bundles $\{E_s\}_{s \in \mathbb{Q}}$, we can show that $\Stab_H^{\mathrm{Geo}}(X) \cong \Stab_H^{\mathrm{Geo}\dagger}(X)$, as detailed in \cite{FLZ}. Consequently, applying \Cref{abelianfold} and \Cref{main1}, we derive the desired result.
\end{proof}

\begin{remark}
\label{refine}
Let $(X,H)$ be a polarized threefold and define the set
\[
\mathfrak{B}:=\left\{(\alpha,\beta,a,b)\in\mathbb R^4\mid\alpha>0,a>\frac{\alpha^2}{6}+\frac{\alpha}{2}|b|\right\}.
\]  
As established in \cite[Section 8]{BMS}, there is a continuous open embedding:
\begin{equation*}
\begin{split}
\Sigma:\gl2\times\mathfrak{B}&\to\stabgh(X)\\
(g,(\alpha,\beta,a,b))&\mapsto\sigma_{\alpha,\beta}^{a,b}[g].
\end{split}
\end{equation*}
Based on \Cref{main1} and \Cref{bound}, we have been able to construct an open subset of $\stabgh(X)$ that refines the result in \cite{BMS}. Note that by \cite[Lemma 3.1]{Re23}, $\Im\Sigma$ and $\Im\Sigma_{\Psi}$ are contractible. In particular, if $\Psi_X(\alpha,\beta,b)\ge\frac{\alpha^2}{6}$ for all $(\alpha,\beta,b)\in\mathbb R_{>0}\times\mathbb R^2$, we have $\Stab^{\Geo\dagger}_H(X)$ is contractible.
\end{remark}

\section{Geometric stability conditions on $\p3$}
\label{sec.gldim}
In this section, we focus on the space of geometric stability conditions on $\p3$. We first demonstrate that the the global dimension of the subspace of stability conditions $\Im\Sigma(\p3)$ constructed in \cite{BMS}  is 3 following the approach outlined in \cite{BMT,BMS,Moz22}, and then, we construct a family of stable vector bundles on $\p3$ in the  large volume limit. 

\subsection{Global dimension of geometric stability conditions on $\p3$}

We first state the main result of this section as follows:
\begin{theorem}
\label{main2}
There is an injection
\[
\iota_3:\Im\Sigma(\p3) \hookrightarrow\mathrm{gldim}^{-1}(3),
\]
where $\Sigma(\p3)$ is the map described in \Cref{refine} for $\p3$.
\end{theorem}

In order to prove \Cref{main2}, we need to apply the following inequality:
\begin{theorem}
  \label{thm:const}
  \cite[Lemma 8.5,Theorem 8.7]{BMS} Let \(\sigma_{\alpha, \beta}^{a, b} = (Z_{\alpha, \beta}^{a, b}, \cA_{\alpha, \beta})\in\Im\Sigma\).  Then, for every \(\sigma\)-semistable object \(E\), we have
  \[Q_{K}^{\beta}(E) := K \overline{\Delta}_{H}(E) + \overline{\nabla}^{\beta}_{H}(E) \geq 0\]
  for \(K \in I^{a,b}_{\alpha}\), where \(I_{\alpha}^{a, b}\) is the open interval such that kernel of \(Z_{\alpha, \beta}^{a, b}\) is negative with respect to \(Q_{K}^{\beta}\). In particular, \(\frac{1}{2}(\alpha^{2} + 6a) \in I_{\alpha}^{a, b}\) for all \(b\).
\end{theorem}

\begin{proof}[Sketch of proof]
  The proof relies on wall-crossing technique. We begin by demonstrating that if an object \(E\) is \(\sigma_{\alpha, \beta}^{a, b}\)-semistable for \(a \gg 0\), then \(E\) is \(\nu_{\alpha, \beta}\)-semistable modulo some torsion part and therefore satisfies the Bogomolov-Gieseker inequality. We then proceed by induction on the value of \(H \ch_{2}^{\beta}(E) - \frac{\alpha^{2}}{2} \ch_{0}^{\beta}(E)\) which is discrete when \(\alpha, \beta\) are rational numbers. For the minimal possible value, there are essentially no subobject in \(\cA_{\alpha, \beta}\) and \(E\) is \(\sigma_{\alpha, \beta}^{a, b}\)-semistable for \(a \gg 0\). Thus, the inequality holds. For the induction step, if \(E\) is \(\sigma_{\alpha, \beta}^{a, b}\)-semistable for \(a \gg 0\). Then the inequality holds. Otherwise, it is semistable for some \(a = A > 0\) but unstable for \(a = A + \epsilon\). We can then consider the Harder-Narasimhan filtration \(0 = E_{0} \subset E_{1} \subset \dots \subset E_{n} = E\) with respect to \(\sigma_{\alpha, \beta}^{A+ \epsilon, b}\), which is independent of small \(\epsilon > 0\). By induction hypothesis, the inequality holds for every Harder-Narasimhan factor. Since \(Q_{K}^{\beta}\) is negative definite with respect to \(Z_{\alpha, \beta}^{a, b}\), we have \(Q_{K}^{\beta}(E) \geq \sum Q_{K}^{\beta}(E_{i}/E_{i - 1}) \geq 0\), and the inequality holds.

  For \(\alpha, \beta\) irrational, we utilize the fact that \(\sigma_{\alpha, \beta}^{a, b}\)-stable object is \(\sigma_{\alpha', \beta'}^{a, b}\)-stable for \((\alpha', \beta')\) in a neighborhood of \((\alpha, \beta)\). This allows us to extend the inequality for stable object follows from the rational case. For semistable case, we use the Jordan-H\"older filtration and use the negative definite property of \(Q_{K}^{\beta}\) again to establish the inequality.
\end{proof}

\Cref{thm:const} yields the following inequality:
\begin{corollary}
\label{Zieq}
  Let \(Z_{t} = Z^{a,b}_{\alpha, \beta - tc}\) with \((\alpha,\beta,a,b)\in\mathfrak{B}\), and \(c \ge 0\). Then we have
  \[\Im (Z_{t}'(E) \overline{Z_{t}(E)})|_{t = 0} \geq 0\]
  if \(E\) is \(Z_{0}\)-semistable.
\end{corollary}

\begin{proof}
  \[Z_{t}'(E)|_{t = 0} = - cH \ch_{2}^{\beta}(E) + bcH^{2} \ch_{1}^{\beta} + ac H^{3} \ch_{0}^{\beta}(E) + i \alpha c H^{2} \ch_{1}^{\beta}(E)\]
 For simplicity, let \(\zeta_{i} := H^{3 - i} \ch_{i}^{\beta}(E)\). Then,
  \begin{eqnarray*}
    \Im (Z_{t}'(E) \overline{Z_{t}(E)})|_{t = 0} & = & c \zeta_{2}^{2} - (ac + \frac{\alpha^{2}c}{2}) \zeta_{0} \zeta_{2} + \frac{\alpha^{2}}{2} bc \zeta_{0} \zeta_{1} + \frac{\alpha^{2} ac}{2} \zeta_{0}^{2} - c \zeta_{1} \zeta_{3} + ac \zeta_{1}^{2}\\
                                                 & = & c(\frac{1}{6}Q_{\frac{1}{2}(\alpha^{2} + 6a)} + \frac{1}{3} \zeta_{2}^{2} - \frac{\alpha^{2}}{3} \zeta_{0}\zeta_{2} + (\frac{a}{2} - \frac{\alpha^{2}}{12} \xi_{1})^{2} + \frac{\alpha^{2}b}{2} \zeta_{0} \zeta_{1} + \frac{\alpha^{2}a}{2} \zeta_{0}^{2}) \\
                                                 & \geq & c (\frac{1}{6}Q_{\frac{1}{2}(\alpha^{2} + 6a)} + \frac{1}{3} \zeta_{2}^{2} - \frac{\alpha^{2}}{3} \zeta_{0} \zeta_{2} + \frac{\abs{b} \alpha}{4} \zeta_{1}^{2} + \frac{\alpha^{2}}{2} \zeta_{0}\zeta_{1} + \frac{\alpha^{2}}{2}(\frac{\alpha^{2}}{6} + \frac{\abs{b} \alpha}{2}) \zeta_{0}^{2}) \\
                                                 & = & c(\frac{1}{6}Q_{\frac{1}{2}(\alpha^{2} + 6a)} + \frac{1}{3}( \zeta_{2} - \frac{\alpha^{2}}{2} \zeta_{0})^{2} + \frac{\abs{b} \alpha}{4}(\xi_{1} \pm \xi_{0})^{2}) \\
    & \geq & 0
  \end{eqnarray*}
  Note that the choice of \(\pm\) in the second to last line depend on the sign of \(b\).
\end{proof}

We now define a functor associated with twisted line bundles on $\p3$:
\begin{definition}
\label{functorL}
Let $L=\mathcal O_{\p3}(c)$ be a line bundle on $\p3$, for some $c\in\mathbb Z$. Then we define an automorphism $\Phi_L\in\mathrm{Aut}(\mathcal D^b(\p3))$ induced by $L$ as follows:
\[
\Phi_L:\mathcal D^b(\p3)\to\mathcal D^{b}(\p3), \quad E\mapsto E\otimes L.
\]
\end{definition}
\begin{lemma}
\label{Phi}
$\Phi_L(\sigma^{a,b}_{\alpha,\beta})=\sigma^{a,b}_{\alpha,\beta+c}$. Here, \(\sigma_{\alpha, \beta}^{a,b} = (Z^{a,b}_{\alpha, \beta}, \cA_{\alpha, \beta}) \in\Stab^{\Geo\dagger}(\p3)\).
\end{lemma}
\begin{proof}
  Notice that
  \[\ch(E \otimes L^{-1}) = e^{-cH} \ch(E) = \ch^{c}(E)\]
  For $E\in K_{\num}(\p3)$,
\begin{equation}
\begin{split}
\Phi_L(Z^{a,b}_{\alpha,\beta}(E))&=Z^{a,b}_{\alpha,\beta}(E\otimes L^{-1})\\
&=Z^{a, b}_{\alpha,\beta+c}(E).
\end{split}
\end{equation}
By definition, $\Phi_L(\mathcal A_{\alpha,\beta})=\langle\mathcal T'_{\alpha,\beta}\otimes L,\mathcal F'_{\alpha,\beta}[1]\otimes L \rangle$, and
$\mathcal A_{\alpha,\beta+c}=\langle\mathcal T'_{\alpha,\beta+c},\mathcal F'_{\alpha,\beta+c}[1] \rangle$, where
\begin{equation*}
\begin{split}
\mathcal T'_{\alpha,\beta+c}&=\{E\in\mathrm{Coh}^{\beta+c}(X): \text{ any quotient } E\twoheadrightarrow G \text{ satisfies } \nu_{\alpha,\beta+c}(G)>0\};\\
\mathcal F'_{\alpha,\beta+c}&=\{E\in\mathrm{Coh}^{\beta+c}(X): \text{ any  subsheaf } F\hookrightarrow E\text{ satisfies } \nu_{\alpha,\beta+c}(F)\le0\}.
\end{split}
\end{equation*}
Let $E\in\mathcal T'_{\alpha,\beta}$ with the following Harder-Narasimhan filtration with respect to $\nu_{\alpha,\beta}$-stability:
\[
0=E_0\subset E_1\subset\cdots\subset E_n=E,
\]
where $F_i=E_i/E_{i-1}$ is $\nu_{\alpha,\beta}$-semistable with $\nu_{\alpha,\beta}(F_i)>\nu_{\alpha,\beta}(F_{i+1})$ for all $i$. Then we have another filtration:
\[
0=E_0\otimes L\subset E_1\otimes L\subset\cdots\subset E_n\otimes L=E\otimes L.
\]
We claim that the new filtration is also a Harder-Narasimhan filtration with respect to $\nu_{\alpha,\beta+c}$-stability, i.e. $F_i'=F_i\otimes L$ is $\nu_{\alpha,\beta+c}$-semistable with $\nu_{\alpha,\beta+c}(F_i)>\nu_{\alpha,\beta+c}(F_{i+1})$ for all $i$. For any $0\ne F\subsetneq F_i'$, we have
\begin{equation}
\begin{split}
\nu_{\alpha,\beta+c}(F)&=\frac{\Im(Z_{\alpha,\beta+c}^{a,b}(F))}{(\alpha H)^2\mathrm{ch}_1^{\beta+c}(F)}\\
&=\frac{\Im(Z_{\alpha,\beta}^{a,b}(F\otimes L^{-1}))}{(\alpha H)^2\mathrm{ch}_1^{\beta}(F\otimes L^{-1})}\\
&=\nu_{\alpha,\beta}(F\otimes L^{-1})\\
&\le\nu_{\alpha,\beta}(F_i)\\
&=\nu_{\alpha,\beta+c}(F_i'),
\end{split}
\end{equation}
and
\[
\nu_{\alpha,\beta+c}(F_i')=\nu_{\alpha,\beta}(F_i)>\nu_{\alpha,\beta}(F_{i+1})=\nu_{\alpha,\beta+c}(F_{i+1}'),
\]
which implies that $F_i'$ is $\nu_{\alpha,\beta+c}$-semistable with $\nu_{\alpha,\beta+c}(F_i)>\nu_{\alpha,\beta+c}(F_{i+1})$ for all $i$. Note that
\[
\nu_{\alpha,\beta+c}(F_n')=\nu_{\alpha,\beta}(F_n)>0.
\]
Then, we have $\mathcal T'_{\alpha,\beta}\otimes L\subset\mathcal T'_{\alpha,\beta+c}$. Similarly, we can show $\mathcal F'_{\alpha,\beta}[1]\otimes L\subset\mathcal F'_{\alpha,\beta+c}[1]$. Therefore, $\Phi_L(\mathcal A_{\alpha,\beta})\subset\mathcal A_{\alpha,\beta+c}$. As both are hearts of bounded t-structure, we obtain $\Phi_L(\mathcal A_{\alpha,\beta})=\mathcal A_{\alpha,\beta+c}$. Hence, we conclude that $\Phi_L(\sigma^{a,b}_{\alpha,\beta})=\sigma^{a,b}_{\alpha,\beta+c}$.
\end{proof}

Based on the discussion above, we can now provide the proof of \Cref{main2}:
\begin{proof}\emph{(of \Cref{main2})}
Let $\sigma^{a,b}_{\alpha,\beta}=(Z^{a,b}_{\alpha,\beta},\mathcal A_{\alpha,\beta})\in\Im\Sigma(\p3)$. For any $x\in\p3$,
\[
\mathrm{Ext}^3(\mathcal O_x,\mathcal O_x)=\mathrm{Hom}(\mathcal O_x,\mathcal O_x)^{\vee}\ne0.
\]
Since $\mathcal O_x$ is $\sigma^{a,b}_{\alpha,\beta}$-stable, we obtain $\mathrm{gldim}(\sigma^{a,b}_{\alpha,\beta})\ge3$. For simplicity, denote by $\phi^{a,b}_{\alpha,\beta}:=\phi_{\sigma^{a,b}_{\alpha,\beta}}$. Let $E,F$ be any $\sigma^{a,b}_{\alpha,\beta}$-semistable objects such that $\phi^{a,b}_{{\alpha,\beta}}(E)<\phi^{a,b}_{{\alpha,\beta}}(F)$. We claim that for any ample line bundle $L=\mathcal O(cH)$ on $\p3$, where $c>0$, $\mathrm{Hom}(F\otimes L,E)=0$.
 Consider the family of stability conditions
\[
(Z_t,\mathcal P_t):=\sigma^{a,b}_{\alpha,\beta-tc}, \quad t\in[0,1],
\]
with phase function $\phi_t$. By \Cref{Zieq}, we have $\Im(Z'_t(F)\cdot\overline{Z_t(F)})\ge0$. Due to \cite[Theorem 3.12]{Moz22}, it follows that $\phi_1^-(F)\ge\phi_0(F)$, i.e. $\phi^{a,b, -}_{\alpha,\beta-c}(F)\ge\phi^{a,b}_{\alpha,\beta}(F)$. According to \Cref{Phi}, the set of semistable objects in $\mathcal A_{\alpha,\beta}$ with respect to $Z^{a,b}_{\alpha,\beta}$ coincide with the set of semistable objects $\mathcal A^{a,b}_{\alpha,\beta-c}$ with respect to $Z^{a,b}_{\alpha,\beta-c}$. Then, we derive
\[
\phi^{a,b}_{\alpha,\beta-c}(F)\ge\phi^{a,b}_{\alpha,\beta}(F)>\phi^{a,b}_{\alpha,\beta}(E)=\phi^{a,b}_{\alpha,\beta-c}(E\otimes L^{-1}).
\]
Hence,
\[
\mathrm{Hom}(F\otimes L,E)\simeq\mathrm{Hom}(F,E\otimes L^{-1})=0.
\]
Now we suppose $\phi^{a,b}_{\alpha,\beta}(F)>\phi^{a,b}_{\alpha,\beta}(E)+3$. By Serre duality,
\[
\mathrm{Hom}(E,F)=\mathrm{Hom}(F\otimes\mathcal O(4),E[3])=0.
\]
The last equality follows from $\phi^{a,b}_{\alpha,\beta}(E[3])<\phi^{a,b}_{\alpha,\beta}(F)$. In conclude, we obtain for any $\sigma^{a,b}_{\alpha,\beta}$-semistable objects $E,F$ with $\phi^{a,b}_{\alpha,\beta}(F)>\phi^{a,b}_{\alpha,\beta}(E)+3$, then $\mathrm{Hom}(E,F)=0$, which implies $\gldim(\sigma^{a,b}_{\alpha,\beta})\le3$.

For a general \(\sigma \in \Im\Sigma(\p3)\), we have \(\sigma =\sigma_{\alpha, \beta}^{a,b}[g] = (Z_{\alpha,\beta}[g], \cA_{\alpha,\beta}[g])\) for some \(g \in \gl2\). We can then consider the family \((Z_{t}, \cP_{t}) := \sigma_{\alpha, \beta}^{a,b}[g]\), and observe that
\[\Im (Z_{t}'(E)\overline{Z_{t}(E)})|_{t = 0} = \abs{g} (Z_{\alpha, \beta - ct}^{a,b,\prime}(E)\overline{Z_{\alpha, \beta - ct}^{a,b}(E)})|_{t = 0} \geq 0\]
for any \(Z_{0}\)-semistable object \(E\). By repeating the above procedure for \(\sigma_{\alpha, \beta}^{a,b}\), we obtain the required result.
\end{proof}

\begin{remark}
According to \cite[Theorem 4.2 and Exercise 2.5]{KOT}, we have $\mathrm{gldim}(\mathcal D^b(\mathbb P^n))\ge n$. The equality follows immediately as there is a stability condition $\sigma$ with $\mathrm{gldim}(\sigma)=n$, where the heart of $\sigma$ is the Beilinson heart $\mathcal H$ with simples
\[
\mathrm{Sim}\,\mathcal H=\mathcal O[n],\mathcal O(1)[n-1],\cdots,\mathcal O(n)
\]
and the central charges $Z(\cO(i)[n-i])=1$. Then we have $\mathfrak{B}$ is a $\mathrm{gldim}$-reachable open subset of $\Stab(\p3)$ by \Cref{main2}.
\end{remark}

\begin{remark}
Extending our method to general projective threefolds is tricky for a couple of reasons. For one, we rely on the Serre functor and need the anticanonical divisor \(-K_X\) to be ample to derive the inequality. Moreover, even in the case of Fano threefolds, it's not always the case that a projective threefold has a Picard rank of \(1\). As a result, the canonical divisor \(K_X\) and a generic divisor \(H\) might not be propotional, which leaves the existence of Bridgeland stability for the form \(\sigma_{\alpha H,\beta H - tK_X}\) up in the air. In addition, it requires some ajustments to the inequalities stated in Theorem \ref{thm:const}.
\end{remark}

\subsection{Large volume limit}
The key technical strategy outlined in \cite{FLZ} for describing the space of stability conditions on polarized abelian threefolds relies on the use of a set of simple semi-homogeneous vector bundles. These bundles are stable with respect to any numerical stability condition, and their slopes form a dense subset of $\mathbb{R}$. In this section, we construct a family of vector bundles whose slopes are dense over $\mathbb R$ and which are stable with repect to stability conditions on $\mathcal{D}^b(\mathbb{P}^3)$ in the large volume limit. This serves as an initial attempt to gain a more profound insight on $\Stab(\p3)$.

Let us consider $\sigma_t=(Z_t,\cP_t)=\sigma_{t\alpha,\beta}\in\Stab^{\Geo\dagger}(\p3)$, for $t>0$. We will denote an object $0\ne E\in\mathcal D^b(\p3)$ as $\sigma_{\infty}$-semistable (resp. stable) (or semistable (resp. stable) in the \emph{large volume limit}) if $E$ is $\sigma_t$-semistable (resp. stable) for $t\gg0$. Similarly, we have the notion of $\nu_{\infty}$-semstable (resp. stable). We denote an object $E\in\cP_{\infty}(I)$, if $E\in\cP_t(I)$ for $t\gg0$. We now review a result that identifies slope stability and $\nu_{\infty}$-stability in the large volume limit under some conditions.

\begin{lemma}[{\cite[Lemma 1.4]{GHS}}]
\label{stonu}
Let $E\in\Cohb(\p3)$ satisfying that $\mu_{\beta}(E)>0$ and $(\mathrm{ch}_0(E),\mathrm{ch}_1(E))$ is primitive. Then $E$ is slope stable if and only if $E$ is $\nu_{\infty}$-stable. 
\end{lemma}

We now establish a connection between $\nu_{\alpha,\beta}$-stability and $\sigma_{\alpha,\beta}^{a,b}$-stable for $a\gg0$:

\begin{lemma}
\label{nutob}
Suppose $\sigma=(Z^{a,b}_{\alpha,\beta},\cA_{\alpha,\beta})\in\stabgp$ and \(E \in \cA_{\alpha, \beta}\) is \(\nu_{\alpha, \beta}\)-stable. Then \(E\) is \(\sigma_{\alpha, \beta}^{a,b}\)-stable for \(a \gg 0\).
\end{lemma}

\begin{proof}
  Let
  \[\rho_{\alpha, \beta}^{a, b} = \frac{\ch^{\beta}_{3} - b \ch^{\beta}_{2} - a \ch^{\beta}_{1}}{\ch^{\beta}_{2} - \frac{\alpha^{2}}{2} \ch_{0}^{\beta}}\]
  be the slope function on \(\cA_{\alpha, \beta}\). Notice that \(\lim_{a \to +\infty}\rho_{\alpha, \beta}^{a, b} = \lim_{a \to +\infty} - a \frac{1}{\nu_{\alpha, \beta}}\).
  Since \(E\) is \(\nu_{\alpha, \beta}\)-stable, we have \(E \in \cA_{\alpha, \beta}\) if \(\nu_{\alpha, \beta} > 0\) and \(E[1] \in \cA_{\alpha, \beta}\) if \(\nu_{\alpha, \beta} \leq 0\).

  In first case, for \(E \twoheadrightarrow F \neq 0\in \cA_{\alpha, \beta}\) and \(F \neq E\), we have \(F \in \cB_{\alpha, \beta} \cap \cA_{\alpha, \beta}\). Since \(E\) is \(\nu_{\alpha, \beta}\)-stable, we have
  \[\nu_{\alpha, \beta}(F) > \nu_{\alpha, \beta}(E) > 0\]
  and therefore
  \[\rho_{\alpha, \beta}^{a, b}(F) > \rho^{a,b}_{\alpha, \beta}(E)\]
  for \(a \gg 0\).

  For the second case, we have \(E[1] \in \cA_{\alpha, \beta}\), and for every \(0 \neq F \subsetneq E[1] \in \cA_{\alpha, \beta}\), we have that \(F \in \cA_{\alpha, \beta} \cap \cB_{\alpha, \beta}[1]\). Since \(E\) is \(\nu_{\alpha, \beta}\)-stable, we have
  \[\nu_{\alpha, \beta}(F[-1]) < \nu_{\alpha, \beta}(E) \leq 0\]
  and therefore
  \[\rho_{\alpha, \beta}^{a, b}(F) < \rho_{\alpha, \beta}^{a, b}(E[1])\]
  for \(a \gg 0\).
\end{proof}

We now proceed to construct the required family of vector bundles:

\begin{proposition}
There exists a family of $\sigma_{\infty}$-stable vector bundles $\{E_{s}\}_{s\in\mathbb Q}\subset\cP_{\infty}((-2,0])$ whose slopes are dense over $\mathbb R$.
\end{proposition}

\begin{proof}
Let $E_{t/r}$ be the vector bundle on $\p3$ determined by the exact sequence
\[
0\to\mathcal O(-1)^t\xrightarrow{M}\cO^{r+t}\to E_{t/r}\to 0,
\]
where $M$ is a general matrix of linear form. By \cite[Theorem 5.4]{CHS}, if $r<(1+\sqrt{3})t$, $V_{t/r}$ is slope stable with $\mu(V_{t/r})=\frac{t}{r}$. Then we obtain a family of slope stable bundles $\{E_{t/r}:\frac{t}{r}\in(\frac{1}{\sqrt{3}+1},+\infty)\cap\mathbb Q\}$ on $\p3$. Denote by $E'_{-t/r}:=\mathbb D(E_{t/r})$, where $\mathbb D$ is the local dualizing  functor on $\mathcal D^b(\p3)$. Then 
\[\Phi_{\mathcal O(1)}\circ\mathbb D\left(\{E'_{t/r}: \frac{t}{r}\in(-\infty,-\frac{1}{\sqrt3+1})\cap\mathbb Q\}\right)\cup\{E_{t/r}:\frac{t}{r}\in(\frac{1}{\sqrt{3}+1},+\infty)\cap\mathbb Q\}\]
forms a collection of slope stable bundles whose slopes cover $\mathbb Q$, where $\Phi_{\mathcal O(1)}$ is the automorphism on $\mathcal D^b(\p3)$ induced by $\cO(1)$  introdcued in \Cref{functorL}. By \Cref{stonu,nutob}, we obtain a family of $\sigma_{\infty}$-stable vector bundles $\{E_s\}_{s\in\mathbb Q}$ with dense slopes over $\mathbb R$. Furthermore, we have 
\begin{equation*}
E_{s}\in
\begin{cases}
\cP_{\infty}((-1,0]), \quad &\mu_{\beta}(E_s)\ge0;\\
\cP_{\infty}((-2,-1]), \quad &\mu_{\beta}(E_s)<0.\\
\end{cases}
\end{equation*}

\end{proof}

\section{Conjecture: Contractibility of \(\stabp\)}
\label{sec:algebr-stab-cond}
The conjecture of contractibility of $\Stab(\mathbb P^n)$ is expected in \cite[Section 6.2]{Qiu18}. In this section, we formulate a conjecture concerning the contractibility of a principal connected component $\mathrm{Stab}^{\dagger}(\p3)\subset\Stab(\p3)$, inspired by \cite{Li17}. 

We start with a brief overview of exceptional objects, and we recmmend \cite{Bon90, GR,KO} for more details. 
\begin{definition}
Let $\cD$ be a $\mathbb C$-linear triangulated category of finite type. An object $E\in\cD$ is called \emph{exceptional} if 
\[
\mathrm{Hom}^i(E,E)=0,\quad \text{ for } i\ne0; \quad \mathrm{Hom}^0(E,E)=\mathbb C.
\]
An ordered collection of exceptional objects $\cE=\{E_1,\dots,E_n\}$ is called an \emph{exceptional collection} if 
\[
\mathrm{Hom}^{\bullet}(E_i,E_j)=0,\quad \text{ for } i>j.
\]
The collection $\cE$ is said to be \emph{strong} if $\mathrm{Hom}^q(E_i,E_j)=0$ for all $i,j$ and $q\ne0$. $\cE$ is denoted as \emph{full}, if $\cE$ generates $\mathcal D$ under homological shifts, cones, and direct summands. 
\end{definition}

Accroding to \cite[Theorem 9.3]{Bon90}, every full exceptional collection on $\cD^b(\mathbb P^n)$ is strong. We proceed to review the construction of algebraic stability conditions with repsect to exceptional collections on $\cD^b(\p3)$.

\begin{proposition}[{\cite[Section 3]{M07a}}]
\label{algstab}
Let $\cE=\{E_1,E_2,E_3,E_4\}$ be a full exceptional collection on $\cD(\p3)$. For any positive real numbers $m_i$ and $\phi_i$, for $i=1,\dots,4$, satisfying:
\[
\phi_j-\phi_i>\frac{(j-i)(j-i+1)}{2}, \quad \emph{\text{ for all }} i<j. 
\]
There exists a unique stability condition $\sigma=(Z,\cP)$ such that 
\begin{enumerate}
\item each $E_j$ is stable with phase $\phi_j$;
\item $Z(E_j)=m_je^{i\pi\phi_j}$.
\end{enumerate}
\end{proposition}

Consider $\cE=\{E_1,\dots,E_4\}$ a full exceptional collection in $\cD^b(\p3)$. We introduce $\Theta_{\cE}$ associated with $\cE$, which consists of all stability conditions as described in \Cref{algstab}, i.e. $\Theta_{\cE}$ is parametrized by 
\[
\{(m_1,\dots,m_4,\phi_1,\dots,\phi_4)\in(\mathbb R_{>0})^4\times\mathbb R^4\mid \phi_j-\phi_i>\frac{(j-i)(j-i+1)}{2}, \, \text{ for all } i<j\}.
\] 
An element in $\Theta_{\cE}$ is referred to as an algebraic stability condition on $\cD^b(\p3)$. We denote $\staba$ as the union of $\Theta_{\cE}$ for all full exceptional collections on $\cD^b(\p3)$. Each subspace $\Theta_{\cE}$ may intersect with others, and it also possesses a unique subset:

\begin{proposition}
Let $\mathcal E=\{E_1,\dots,E_4\}$ be a full exceptional collection. Then
\[
\Theta^{\ast}_{\cE}:=\{\sigma\in\Theta_{\cE}\mid\phi_{i+1}-\phi_{i}\ge1, \text{ for } 1\le i\le 3\}
\]
constitutes the unique component for $\Theta_{\cE}$, i.e. for any other full exceptional collection $\cE'$, we have $\Theta^{\ast}_{\cE}\cap\Theta^{\ast}_{\cE'}=\emptyset$.
\end{proposition}
\begin{proof}
The proof is based on \cite[Lemma 2.4]{Li17}. We claim that, for $\sigma\in\Theta^{\ast}_{\cE}$, the only $\sigma$-stable objects are $E_i[m]$ for $1\le i\le 4$ and $m\in\mathbb Z$. Let $\sigma\in\Theta_{\cE}^{\ast}$ and  $F$ be a $\sigma$-stable object. We may assume that it is in the heart $\langle E_1[a_1],\dots,E_4[a_4]\rangle$, where $a_i\in\mathbb Z$ such that $a_i-a_{i+1}\ge1$ for $i=1,2,3$. Suppose we have the following sequence:
\[
E_i^{\oplus n_i}[a_i]\to F\to E_{j}^{\oplus n_j}[a_j],
\]
for $n_i,n_j\in\mathbb Z_{\ge0}$, $1\le i,j\le 4$ and $i\ne j$. We then have the following cases:
\begin{enumerate}
\item $i\le j-1$: $F$ is $E_{i}[a_i]$ or $E_{j}[a_j]$ since $\mathrm{Hom}^{\bullet}(E_j,E_i)=0$;
\item $i\ge j+1$: If $a_j=a_i+1$, we have $\phi(E_i[a_i])\ge\phi([E_j[a_j]])$ and thus $F$ is $E_{i}[a_i]$ or $E_{j}[a_j]$ due to the stability of $F$. If $a_j\ge a_i+2$, $F$ is either $E_i[a_i]$ or $E_j[a_j]$ as
\[
\mathrm{Ext}^1(E_j[a_j],E_i[a_i])=\mathrm{Ext}^1(E_i[a_i],E_j[a_j])=0.
\]
\end{enumerate}
We deduce that $F$ is in the form of $E_i[m]$ for $1\le i\le 4$ and $m\in\mathbb Z$, since $F$ can be obtained by a sequence of extensions in $\{E_1[a_1],\dots,E_n[a_4]\}$. 

Now, we only need to consider the full exceptional collection $\cE'$ which is equivalent to $\cE$ up to shift. Using the metric on $\Stab(\p3)$, we can select an open neighborhood $U_{\cE}^{\ast}$ of $\Theta_{\cE}^{\ast}$ such that $U_{\cE}^{\ast}$ does not intersect with $U_{\cE'}^{\ast}$. 
\end{proof}

Inspired by \cite{Li17}, we expect that whole space of $\staba$ cab be contracted to $\stabg$ along the boundary of $\stabg$. The following result shows us the boundary of finitely many $\Theta_{\cE}$ is contained in $\staba$:
\begin{proposition}[{\cite[Theorem 4.7]{M07b}}]
Let $\cE$ be a full exceptional collection on $\cD^b(\p3)$. We have 
\[
\partial\Theta_{\cE}\subset\staba.
\]
\end{proposition}

Denote by $\stabp$ the {principal connected component} of $\Stab(\p3)$ that contains $\Im\Sigma_{\Psi}$. Note that $\Im\Sigma_{\Psi}$ is a contractible open subset in $\Stab(\p3)$. The following result provides a characterization of the stability conditions on $\partial\Im\Sigma_{\Psi}$:

\begin{proposition}
\label{partialgeo}
Let $(X,H)$ be a polarized threefold. For a stability condition $\sigma_{\alpha,\beta}^{a,b}=(Z_{\alpha,\beta}^{a,b},\cA_{\alpha,\beta})\in\overline{\Im\Sigma_{\Psi}}$, if there is a $\nu_{\alpha,\beta}$-semistable object $0\ne E\in\mathrm{Ker}\,Z_{\alpha,\beta}^{a,b}\cap\Cohb(X)$, then $\sigma_{\alpha,\beta}^{a,b}\in\partial\Im\Sigma_{\Psi}$. Moreover, if $\Psi_X(\alpha,\beta,b)\ge\frac{\alpha^2}{6}$, $\sigma\in\partial\Stab^{\Geo}(X)$.
\end{proposition}
\begin{proof}
By definition, we have
\[(\alpha,\beta,a,b)\in\mathbb R_{\ge0}\times\mathbb R^3, a\ge\max\{\frac{\alpha^2}{6},\Psi_X(\alpha,\beta,b)\}.\]
 Since $E\in\mathrm{Ker}\,Z_{\alpha,\beta}^{a,b}$, we have 
\begin{equation*}
\begin{split}
\Re(Z(E))&=\ch_3^{\beta}(E)-bH\ch_2^{\beta}(E)-aH^2\ch_1^{\beta}(E)=0,\\
\Im(Z(E))&=H\ch_2^{\beta}(E)-\frac{\alpha^2}{2}H^3\ch_0^{\beta}(E)=0.
\end{split}
\end{equation*}
Then, we obtain
\[
\frac{\ch_3^{\beta}(E)-bH\ch_2^{\beta}(E)}{H^2\ch_1^{\beta}(E)}\le\Psi_{\p3}(\alpha,\beta,b)\le a=\frac{\ch_3^{\beta}(E)-bH\ch_2^{\beta}(E)}{H^2\ch_1^{\beta}(E)},
\]
which implies $a=\Psi_{X}(\alpha,\beta,b)$, i.e. $\sigma_{\alpha,\beta}^{a,b}\in\partial\Im\Sigma_{\Psi}$. The moreover part follows from \Cref{main1}.
\end{proof}
Note that \Cref{partialgeo} remains valid under the $\gl2$-action. We now formulate the principal conjecture in this section:
\begin{conjecture}
\label{mian3}
The principal connected component $\stabp\subset\Stab(\p3)$ is the union of geometric and algebraic stability conditions and is contractible.
\[
\stabp=\stabg\bigcup\staba.
\]
\end{conjecture}

\begin{remark}
The principal connected component $\stabp$ has also been considered in \cite{OPT} for the purpose of studying the reduced Donaldson-Thomas invariants. To delve into this conjecture, a more comprehensive understanding of full exceptional collections in $\cD^b(\p3)$ is required, as well as a complete description of $\stabg$. As observed in \cite{FLLQ}, the global dimension of geometric stability conditions in $\Stab^{\Geo}(\mathbb P^2)$ is not always 2. By analogy, it is expected that the global dimension of  $\stabg$ will not consistently be 3. By \Cref{main2}, the global dimension of geometric stability conditinos constructed by \cite{BMS} is 3, which then is not expected to be $\stabg$. However, we expect the space $\Im\Sigma_{\Psi}$, as constructed in \Cref{main1}, corresponds to $\stabg$. \Cref{partialgeo} provides some evidence to support this expectation. Furthermore, in accordance with \cite{FLLQ}, we anticipate that $\stabp$ is contractible with respect to the global dimension function.
\end{remark}


\appendix

\end{document}